\documentclass[10pt]{amsart}
\usepackage{amscd,amsmath,amsthm,amsfonts,amssymb,latexsym,graphicx,color,geometry,hyperref}
\usepackage{amsmath}
\usepackage{amsfonts}
\usepackage{amssymb}
\usepackage{graphicx}%
\providecommand{\U}[1]{\protect\rule{.1in}{.1in}}

\newtheorem{theorem}{Theorem}
\newtheorem*{theorem*}{Theorem}
\theoremstyle{plain}

\newtheorem{corollary}{Corollary}

\newtheorem{example}{Example}

\numberwithin{equation}{section}
\newtheorem*{ClassInc}{Linear Inclusion Theorem}
\newtheorem*{IncNorm}{Inclusion on $\ell_p$ spaces}
\newtheorem*{Mink}{Minkowski's inequality}

\begin{document}
\title[Anisotropic Regularity Principle in sequence spaces and applications]{Anisotropic Regularity Principle in sequence spaces and applications}
\author[Albuquerque]{Nacib Albuquerque}
\address[N. Albuquerque]{Departamento de Matem\'{a}tica, Universidade Federal da
Para\'{\i}ba, 58.051-900, Jo\~{a}o Pessoa -- PB, Brazil.}
\email{ngalbqrq@gmail.com and ngalbuquerque@pq.cnpq.br}
\author[Rezende]{Lisiane Rezende}
\address[L. Rezende]{Departamento de Matem\'{a}tica, Universidade Federal da
Para\'{\i}ba, 58.051-900, Jo\~{a}o Pessoa -- PB, Brazil.}
\email{lirezendestos@gmail.com}

\begin{abstract}
We refine a recent technique introduced by Pellegrino, Santos, Serrano and
Teixeira and prove a quite general anisotropic regularity principle in sequence spaces. As applications we generalize previous results of several authors regarding Hardy--Littlewood inequalities for multilinear forms. 

\end{abstract}
\maketitle
\tableofcontents

\section{Introduction}

Regularity techniques are crucial in many fields of pure and applied sciences. Recently, Pellegrino, Teixeira, Santos and Serrano addressed a regularity problem in sequence spaces with deep connections with the Hardy--Littlewood inequalities. In this paper we follow this vein, exploring the ideas from the Regularity Principle proven by Pellegrino \textit{et al.} \cite{PSST} and providing a couple of applications. 

The paper is organized as follows. In Section 2, borrowing ideas from \cite{PSST} we prove an anisotropic inclusion theorem for summing operators that will be useful along the paper. The techniques and arguments explored in Section 2 paves the way to the statement of a kind of anisotropic regularity principle for sequence spaces/series, in Section 3, completing results from \cite{PSST}. In the final section the bulk of results are combined to prove new Hardy--Littlewood inequalities for multilinear operators. 

\subsection{Summability of multilinear operators}

The theory of multiple summing multilinear
mappings was introduced in \cite{matos03, pg}; this class is certainly one of the most useful and fruitful multilinear generalizations
of the concept of absolutely summing linear operators, with important
connections with the Bohnenblust--Hille and Hardy-Littlewood inequalities and its applications in applied sciences. For recent results on absolutely summing
operators and these classical inequalities we refer to
\cite{blasco,cav.quaest,maia} and the references therein.

The Hardy--Littlewood inequalities have its starting point in 1930 with
Littlewood's $4/3$ inequality \cite{l}. In 1934 Hardy and Littlewood extended
Littlewood's $4/3$ inequality \cite{hardy} to more general sequence spaces.
Both results are for bilinear forms. In 1981 Praciano-Pereira \cite{pra} extended these
results to the $m$-linear setting and recently various authors have revisited
this subject. In 2016 Dimant and Sevilla-Peris \cite{dimant} proved the
following inequality: for all positive integers $m$ and all $m<p\leq2m$ we have%

\begin{equation}
\left(  \sum_{j_{1},\cdots,j_{m}=1}^{\infty}\left\vert T(e_{j_{1}}%
,\cdots,e_{j_{m}})\right\vert ^{\frac{p}{p-m}}\right)  ^{\frac{p-m}{p}}%
\leq\left(  \sqrt{2}\right)  ^{m-1}\left\Vert T\right\Vert \label{m88}%
\end{equation}
for all continuous $m$-linear forms $T:\ell_{p}\times\cdots\times\ell
_{p}\rightarrow\mathbb{K}$. It is also proved that the exponent $\frac{p}%
{p-m}$ cannot be improved. However, this optimality seems to be just apparent, as remarked in
some previous works (see \cite{APNSHL, cav.quaest}). Following these lines,
the exponent can be potentially improved in the anisotropic viewpoint. In
order to do that, the theory of summing operators shall play a fundamental role.

Along the years, somewhat puzzling inclusion results for multilinear summing
operators were obtained \cite{pams,complu,pg}. In this note we prove an
inclusion result for multiple summing operators generalizing recent approaches
of Pellegrino, Santos, Serrano and Teixeira \cite{PSST} and Bayart
\cite{bayart}. It is interesting to note that, \textit{en passant}, a sharper
version of the Hardy--Littlewood inequalities for $m$-linear forms, not
encompassed by several recent attempts (for instance, \cite{APNSHL}), is
provided as application. In fact, we show that under the same hypothesis of (\ref{m88}) we
have%
\begin{equation}
\left(  \sum_{j_{1}=1}^{\infty}\left(  \ldots\left(  \sum_{j_{m}=1}^{\infty
}\left\Vert T\left(  e_{j_{1}},\dots,e_{j_{m}}\right)  \right\Vert ^{s_{m}%
}\right)  ^{\frac{s_{m-1}}{s_{m}}}\dots\right)  ^{\frac{s_{1}}{s_{2}}}\right)
^{\frac{1}{s_{1}}}\leq\left(  \sqrt{2}\right)  ^{m-1}\Vert T\Vert\label{87654}%
\end{equation}
with
\[
s_{1}=\frac{p}{p-m},
\dots,s_{m}=\frac{2mp}{mp+p-2m},
\]
which is quite better than (\ref{m88}). Despite the huge advance recently
obtained on this direction in \cite{APNSHL}, our results and techniques were
not encompassed by its techniques.

Throughout this paper $X,Y$ shall stand for Banach spaces over the scalar field
$\mathbb{K}$ of real or complex numbers. The topological dual of $X$ and its
unit closed ball are denoted by $X^{\prime}$ and $B_{X^{\prime}}$,
respectively. For $r,p\geq1$, a linear operator $T:X\rightarrow Y$ is said
$(r;p)$-summing if there exists a constant $C>0$ such that
\[
\left(  \sum_{j=1}^{\infty}\left\Vert T(x_{j})\right\Vert ^{r}\right)
^{\frac{1}{r}}\leq C\left\Vert (x_{j})_{j\in\mathbb{N}}\right\Vert _{w,p}%
\]
for any \emph{weakly $p$-summable} vector sequence $(x_{j})_{j\in\mathbb{N}%
}\in\ell_{p}^{w}(X)$, where
\[
\ell_{p}^{w}(X):=\left\{  (x_{j})_{j\in\mathbb{N}}\in X^{\mathbb{N}%
}:\,\left\Vert (x_{j})_{j\in\mathbb{N}}\right\Vert _{w,p}:=\sup_{\varphi\in
B_{X^{\prime}}}\left(  \sum_{j=1}^{\infty}\left\vert \varphi(x_{j})\right\vert
^{p}\right)  ^{\frac{1}{p}}<\infty\right\}  .
\]
For $\mathbf{p}:=(p_{1},\dots,p_{m})\in\lbrack1,\infty]^{m}$, using the
definitions of mixed norm $L_{\mathbf{p}}$ spaces from \cite{benedek}, the
mixed norm sequence space
\[
\ell_{\mathbf{p}}(X):=\ell_{p_{1}}\left(  \ell_{p_{2}}\left(  \cdots\left(
\ell_{p_{m}}\left(  X\right)  \right)  \cdots\right)  \right)
\]
gathers all multi-index vector valued matrices $\mathbf{x}:=\left(
x_{\mathbf{j}}\right)  _{\mathbf{j}\in\mathbb{N}^{m}}$ with finite
$\mathbf{p}$-norm; here $\mathbf{j}:=(j_{1},\dots,j_{m})$ stands for a
multi-index as usual. Notice that each norm $\Vert\cdot\Vert_{p_{k}}$ is taken
over the index $j_{k}$ and that each index $j_{k}$ is related to the
$\Vert\cdot\Vert_{p_{k}}$ norm. For instance, when ${\mathbf{p}}\in
\lbrack1,\infty)^{m}$ a vector matrix $\mathbf{x}$ belongs to $\ell
_{\mathbf{p}}(X)$ if, and only if,
\[
\left\Vert {\mathbf{x}}\right\Vert _{\mathbf{p}}:=\left(  \sum_{j_{1}%
=1}^{\infty}\left(  \sum_{j_{2}=1}^{\infty}\left(  \cdots\left(  \sum
_{j_{m-1}=1}^{\infty}\left(  \sum_{j_{m}=1}^{\infty}\left\Vert x_{\mathbf{j}%
}\right\Vert ^{p_{m}}\right)  ^{\frac{p_{m-1}}{p_{m}}}\right)  ^{\frac
{p_{m-2}}{p_{m-1}}}\dots\right)  ^{\frac{p_{2}}{p_{3}}}\right)  ^{\frac{p_{1}%
}{p_{2}}}\right)  ^{\frac{1}{p_{1}}}<\infty.
\]
Over the last years, many different generalizations of the theory of
absolutely and multiple summing operators were obtained. A natural anisotropic
approach to multiple summing operators is the following: For $\mathbf{r}%
,\mathbf{p}\in\lbrack1,+\infty)^{m}$, a multilinear operator $T:X_{1}%
\times\cdots\times X_{m}\longrightarrow Y$ is said to be multiple
$(\mathbf{r},\mathbf{p})$-summing if there exists a constant $C>0$ such that
for all sequences $x^{k}:=(x_{j}^{k})_{j\in\mathbb{N}}\in\ell_{p_{k}}%
^{w}(X_{k}),\,k=1,\dots,m$,
\[
\left\Vert \left(  T(x_{\mathbf{j}})\right)  _{\mathbf{j}\in\mathbb{N}^{m}%
}\right\Vert _{\mathbf{r}}:=\left(  \sum_{j_{1}=1}^{\infty}\left(
\ldots\left(  \sum_{j_{m}=1}^{\infty}\left\Vert T\left(  x_{\mathbf{j}%
}\right)  \right\Vert ^{r_{m}}\right)  ^{\frac{r_{m-1}}{r_{m}}}\dots\right)
^{\frac{r_{1}}{r_{2}}}\right)  ^{\frac{1}{r_{1}}}\leq C\prod_{k=1}^{m}\Vert
x^{k}\Vert_{w,p_{k}},
\]
where $T(x_{\mathbf{j}}):=T\left(  x_{j_{1}}^{1},\dots,x_{j_{m}}^{m}\right)
$. The class of all multiple $(\mathbf{r},\mathbf{p})$-summing operators is a
Banach space with the norm defined by the infimum of all previous constants
$C>0$. This norm is denoted by $\pi_{(\mathbf{r},\mathbf{p})}(\cdot)$ and the
space that gathers all such operators by $\Pi_{(\mathbf{r};\mathbf{p})}%
^{m}(X_{1},\ldots,X_{m};Y)$. When $r_{1}=\cdots=r_{m}=r,\,s_{1}=\cdots
=s_{m}=s$ we simply write $(r;\mathbf{p}),\,(\mathbf{r};s)$, respectively.




\subsection{Inclusion Theorems}

Basic results from the theory of summing operators are inclusion theorems. For
linear operators, it is folklore that $p$-summability implies $q$-summability
whenever $1\leq p\leq q$. More generally, although basic, the following is
quite useful (see \cite{DJT}).

\begin{ClassInc}
If $s \geq r,\, q \geq p$ and $\frac{1}{p} - \frac{1}{r} \leq\frac{1}{q} -
\frac{1}{s}$, then every absolutely $(r;p)$-summing linear operator is
absolutely $(s;q)$-summing
\end{ClassInc}

Throughout the development of the theory, inclusion theorems reveals as
challenging problems (see, for instance, \cite{pg}). In \cite[Proposition
3.4]{PSST}, the authors proved the followimg  multilinear inclusion result:

\begin{theorem}
[Pellegrino, Santos, Serrano and Teixeira]\label{Pellegrino} Let $m$ be a
positive integer, $r,p,q\in\lbrack1,+\infty)$ be such that $q\geq p$ and
\[
\frac{1}{r}-\frac{m}{p}+\frac{m}{q}>0.
\]
Then
\[
\Pi_{(r;p)}^{m}\left(  X_{1},\dots,X_{m};Y\right)  \subset\Pi_{(s;q)}%
^{m}\left(  X_{1},\dots,X_{m};Y \right)  ,
\]
for any Banach spaces $X_{1},\dots,X_{m}$, with
\[
\frac{1}{s}-\frac{m}{q}=\frac{1}{r}-\frac{m}{p},
\]
and the inclusion operator has norm $1$.
\end{theorem}

Independently, F. Bayart in \cite[Theorem 1.2]{bayart} obtained a more general
version. For $\mathbf{p}\in\lbrack1,+\infty
]^{m}$ and each $k\in\{1,\dots,m\}$ , we define $\left\vert \frac
{1}{\mathbf{p}}\right\vert _{\geq k}:=\frac{1}{p_{k}}+\cdots+\frac{1}{p_{m}}$.
When $k=1$ we write $\left\vert \frac{1}{\mathbf{p}}\right\vert $ instead of
$\left\vert \frac{1}{\mathbf{p}}\right\vert _{\geq1}$

\begin{theorem}
[Bayart]\label{Bayart} Let $m$ be a positive integer, $r,s\in\lbrack
1,+\infty),\,\mathbf{p},\mathbf{q}\in\lbrack1,+\infty)^{m}$ are such that
$q_{k}\geq p_{k},\,k=1,\dots,m$ and
\[
\frac{1}{r}-\left\vert \frac{1}{\mathbf{p}}\right\vert +\left\vert \frac
{1}{\mathbf{q}}\right\vert >0.
\]
Then
\[
\Pi_{(r;\mathbf{p})}^{m}\left(  X_{1},\dots,X_{m};Y\right)  \subset
\Pi_{(s;\mathbf{q})}^{m}\left(  X_{1},\dots,X_{m};Y\right)  ,
\]
for any Banach spaces $X_{1},\dots,X_{m}$, with
\[
\frac{1}{s}-\left\vert \frac{1}{\mathbf{q}}\right\vert =\frac{1}{r}-\left\vert
\frac{1}{\mathbf{p}}\right\vert .
\]

\end{theorem}

In the next section we prove the following inclusion theorem; the techniques are inspired by \cite{PSST} and contained in the proof of the forthcoming Regularity Principle, in Section 3:

\begin{theorem}
\label{Thmain} Let $m$ be a positive integer, $r\geq1,\, \mathbf{s}%
,\mathbf{p},\mathbf{q }\in[1,+\infty)^{m}$ are such that $q_{k} \geq p_{k}$,
for $k=1,\dots,m$ and
\[
\frac{1}{r} - \left|  \frac{1}{\mathbf{p}}\right|  +\left|  \frac
{1}{\mathbf{q}}\right|  > 0.
\]
Then
\[
\Pi_{(r;\mathbf{p})}^{m} \left(  X_{1}, \dots, X_{m}; Y\right)  \subset
\Pi_{(\mathbf{s};\mathbf{q})}^{m} \left(  X_{1}, \dots, X_{m}; Y\right)  ,
\]
for any Banach spaces $X_{1},\dots,X_{m}$, with
\[
\frac{1}{s_{k}} - \left|  \frac{1}{\mathbf{q}}\right|  _{\geq k} = \frac{1}{r}
- \left|  \frac{1}{\mathbf{p}}\right|  _{\geq k},
\]
for each $k \in\{1, \dots, m\}$, and the inclusion operator has norm $1$.
\end{theorem}

\section{The new Inclusion Theorem}


Despite of the general status of the result, only basic facts are used along its proof. The first one is the classical linear inclusion.
We need other standard inclusion type result that we write for future reference.

\begin{IncNorm}
For $q \geq p > 0,\, \| \cdot\|_{q} \leq\| \cdot\|_{p}$.
\end{IncNorm}

The last ingredient is a corollary of one of the many versions of Minkowski's
inequality (see \cite[Corollary 5.4.2]{G}):

\begin{Mink}
For any $0 < p \leq q < \infty$ and for any scalar matrix $(a_{ij}%
)_{i,j\in\mathbb{N}}$,
\[
\left(  \sum_{i=1}^{\infty} \left(  \sum_{j=1}^{\infty} |a_{ij}|^{p}\right)
^{q/p}\right)  ^{1/q} \leq\left(  \sum_{j=1}^{\infty} \left(  \sum
_{i=1}^{\infty} |a_{ij}|^{q}\right)  ^{p/q} \right)  ^{1/p}.\label{mink_ineq}%
\]

\end{Mink}

\subsection{The proof of Theorem \ref{Thmain}}

The argument is inspired on the Regularity Principle of \cite[Theorem
2.1]{PSST}. We will proceed by induction on $m$. The initial case bilinear is
a straightforward application of classical inclusion of linear operators and
$\ell_{p}$ spaces. The ideas used are revealed in the case $m=3$, thus we it
discuss in details. Let $T\in\Pi_{(r;\mathbf{p})}^{3}\left(  X_{1},X_{2}%
,X_{3};Y\right)  $. Then there exists a constant $C>0$ such that
\begin{equation}
\left(  \sum_{j_{3}=1}^{\infty}\left(  \sum_{j_{1},j_{2}=1}^{\infty}\left\Vert
T\left(  x_{\mathbf{j}}\right)  \right\Vert ^{r}\right)  ^{\frac{1}{r}\cdot
r}\right)  ^{\frac{1}{r}}=\left\Vert \left(  T(x_{\mathbf{j}})\right)
_{\mathbf{j}\in\mathbb{N}^{3}}\right\Vert _{r}\leq C\prod_{k=1}^{3}\Vert
x^{k}\Vert_{w,p_{k}}, \label{Op.rp.summing}%
\end{equation}
for all sequences $x^{k}=(x_{j}^{k})_{j\in\mathbb{N}}\in\ell_{p_{k}}^{w}%
(X_{k}),\,k=1,2,3$. Let $x^{k}\in\ell_{p_{k}}^{w}(X_{k})$ with $k=1,2$ fixed.
Defining $v_{3}:X_{3}\longrightarrow\ell_{(r,r)}(Y)$ by
\[
v_{3}(x_{3})=\left(  T\left(  x_{j_{1}}^{1},x_{j_{2}}^{2},x_{3}\right)
\right)  _{j_{1},j_{2}\in\mathbb{N}},
\]
for all $x_{3}\in X_{3}$. By (\ref{Op.rp.summing}) we obtain, for all
$x^{3}\in\ell_{p_{3}}^{w}(X_{3})$,
\[
\left(  \sum_{j_{3}=1}^{\infty}\left\Vert v_{3}\left(  x_{j_{3}}^{3}\right)
\right\Vert ^{r}\right)  ^{\frac{1}{r}}\leq C_{3}\left\Vert x^{3}\right\Vert
_{w,p_{3}},
\]
with $C_{3}=C\prod_{k=1}^{2}\left\Vert x^{k}\right\Vert _{w,p_{k}}$, i.e.,
$v_{3}\in\Pi_{(r;p_{3})}\left(  X_{3};\ell_{(r,r)}(Y)\right)  $. The linear inclusion  \cite[Theorem 10.4]{DJT} lead us to $v_{3}\in\Pi
_{(s_{3};q_{3})}\left(  X_{3};\ell_{(r,r)}(Y)\right)  $ with $q_{3}\geq
p_{3},\,s_{3}\geq r$ such that
\[
\frac{1}{p_{3}}-\frac{1}{r}\leq\frac{1}{q_{3}}-\frac{1}{s_{3}}.
\]
Let us take $\frac{1}{s_{3}}=\frac{1}{r}-\frac{1}{p_{3}}+\frac{1}{q_{3}}>0$.
Applying norm inclusion on $\ell_{p}$ we obtain
\begin{align}
\left(  \sum_{j_{2}=1}^{\infty}\left(  \sum_{j_{1}=1}^{\infty}\left(
\sum_{j_{3}=1}^{\infty}\left\Vert T\left(  x_{\mathbf{j}}\right)  \right\Vert
^{s_{3}}\right)  ^{\frac{s_{3}}{s_{3}}}\right)  ^{\frac{s_{3}}{s_{3}}}\right)
^{\frac{1}{s_{3}}}  &  =\left(  \sum_{j_{3}=1}^{\infty}\left(  \sum_{j_{2}%
=1}^{\infty}\left(  \sum_{j_{1}=1}^{\infty}\left\Vert T\left(  x_{\mathbf{j}%
}\right)  \right\Vert ^{s_{3}}\right)  ^{\frac{s_{3}}{s_{3}}}\right)
^{\frac{s_{3}}{s_{3}}}\right)  ^{\frac{1}{s_{3}}}\nonumber\\
&  \leq\left(  \sum_{j_{3}=1}^{\infty}\left(  \sum_{j_{1},j_{2}=1}^{\infty
}\left\Vert T\left(  x_{\mathbf{j}}\right)  \right\Vert ^{r}\right)
^{\frac{s_{3}}{r}}\right)  ^{\frac{1}{s_{3}}}\nonumber\\
&  \leq C\left\Vert x^{3}\right\Vert _{w,q_{3}}\prod_{k=1}^{2}\left\Vert
x^{k}\right\Vert _{w,p_{k}}. \label{Op.sm.summing}%
\end{align}

Fixing $x^{1} \in\ell_{p_{1}}^{w}(X_{1}),\, x^{3} \in\ell_{q_{3}}^{w}(X_{3})$
and defining $v_{2}: X_{2} \longrightarrow\ell_{(s_{3},s_{3})}(Y)$ by
\[
v_{2}(x_{2}) = \left(  T\left(  x_{j_{1}}^{1}, x_{2},x_{j_{3}}^{3}\right)
\right)  _{j_{1}, j_{3} \in\mathbb{N}},
\]
we observe that (\ref{Op.sm.summing}) leads us to
\[
\left(  \sum_{j_{2}=1}^{\infty} \left\|  v_{2}\left(  x_{j_{2}}^{2}\right)
\right\|  ^{s_{3}} \right)  ^{\frac{1}{s_{3}}} \leq C_{2} \left\|
x^{2}\right\|  _{w,p_{2}},
\]
for all $x^{2} \in\ell_{p_{2}}^{w}(X_{2})$ with $C_{2} = C \left\|
x^{3}\right\|  _{w,q_{3}} \left\|  x^{1}\right\|  _{w,p_{1}}$, i.e., $v_{2}
\in\Pi_{(s_{3};p_{2})}\left(  X_{2};\ell_{(s_{3}, s_{3})}(Y)\right)  $. By the linear inclusion theorem, $v_{2} \in\Pi_{(s_{2};q_{2})}\left(  X_{2}%
;\ell_{(s_{3}, s_{3})}(Y)\right)  $ with $q_{2} \geq p_{2}, \, s_{2} \geq
s_{3}$ and $\frac{1}{p_{2}} - \frac{1}{s_{3}} \leq\frac{1}{q_{2}} - \frac
{1}{s_{2}}$, we get
\begin{align*}
\left(  \sum_{j_{2}=1}^{\infty} \left(  \sum_{j_{1}=1}^{\infty} \left(
\sum_{j_{3}=1}^{\infty} \left\|  T\left(  x_{\mathbf{j}}\right)  \right\|
^{s_{3}} \right)  ^{\frac{s_{3}}{s_{3}}} \right)  ^{\frac{s_{2}}{s_{3}}}
\right)  ^{\frac{1}{s_{2}}} \leq C_{2} \left\|  x^{2}\right\|  _{w,q_{2}} = C
\left\|  x^{1}\right\|  _{w,p_{1}} \prod_{k=2}^{3}\left\|  x^{k}\right\|
_{w,q_{k}}.
\end{align*}
Taking
\[
\frac{1}{s_{2}} = \frac{1}{s_{3}} -\frac{1}{p_{2}} + \frac{1}{q_{2}} =
\frac{1}{r} - \frac{1}{p_{2}} -\frac{1}{p_{3}} + \frac{1}{q_{2}} + \frac
{1}{q_{3}},
\]
since $s_{3} \leq s_{2}$, by using norm inclusion on $\ell_{p}$, we have
\begin{align}
\left(  \sum_{j_{1}=1}^{\infty} \left(  \sum_{j_{2}=1}^{\infty} \left(
\sum_{j_{3}=1}^{\infty} \left\|  T\left(  x_{\mathbf{j}}\right)  \right\|
^{s_{3}} \right)  ^{\frac{s_{2}}{s_{3}}} \right)  ^{\frac{s_{2}}{s_{2}}}
\right)  ^{\frac{1}{s_{2}}}  &  = \left(  \sum_{j_{2}=1}^{\infty} \left(
\sum_{j_{1}=1}^{\infty} \left(  \sum_{j_{3}=1}^{\infty} \left\|  T\left(
x_{\mathbf{j}}\right)  \right\|  ^{s_{3}} \right)  ^{\frac{s_{2}}{s_{3}}}
\right)  ^{\frac{s_{2}}{s_{2}}} \right)  ^{\frac{1}{s_{2}}}\nonumber\\
&  \leq\left(  \sum_{j_{2}=1}^{\infty} \left(  \sum_{j_{1}=1}^{\infty} \left(
\sum_{j_{3}=1}^{\infty} \left\|  T\left(  x_{\mathbf{j}}\right)  \right\|
^{s_{3}} \right)  ^{\frac{s_{3}}{s_{3}}} \right)  ^{\frac{s_{2}}{s_{3}}}
\right)  ^{\frac{1}{s_{2}}}\nonumber\\
&  \leq C \left\|  x^{1}\right\|  _{w,p_{1}} \prod_{k=2}^{3}\left\|
x^{k}\right\|  _{w,q_{k}}. \label{Op.s1p.summing}%
\end{align}

Now let us fix $x^{k} \in\ell_{q_{k}}^{w}(X_{k})$ with $k = 2, 3$ and let us
define, for all $x_{1} \in X_{1}$,
\[
v_{1}(x_{1}) = \left(  T\left(  x_{1}, x_{j_{2}}^{2}, x_{j_{3}}^{3}\right)
\right)  _{j_{2}, j_{3} \in\mathbb{N}}.
\]
Thus $v_{1} \in\Pi_{(s_{2};p_{1})} \left(  X_{1};\ell_{(s_{2},s_{3})}(Y)
\right)  $. By combining (\ref{Op.s1p.summing}) and the linear
inclusion theorem, we get that $v_{1} \in\Pi_{(s_{1};q_{1})} \left(  X_{1};
\ell_{(s_{2},s_{3})}(Y) \right)  $ with $q_{1} \geq p_{1}$ and $s_{1} \geq
s_{2}$ such that
\[
\frac{1}{p_{1}} - \frac{1}{s_{2}} \leq\frac{1}{q_{1}} - \frac{1}{s_{1}}.
\]
By choosing $\frac{1}{s_{1}} = \frac{1}{s_{2}} - \frac{1}{p_{1}} + \frac
{1}{q_{1}} = \frac{1}{r} - \sum_{k=1}^{3}\frac{1}{p_{k}} + \sum_{k=1}^{3}
\frac{1}{q_{k}} > 0$, we have
\begin{align*}
\left(  \sum_{j_{1}=1}^{\infty} \left(  \sum_{j_{2}=1}^{\infty} \left(
\sum_{j_{3}=1}^{\infty} \left\|  T\left(  x_{\mathbf{j}}\right)  \right\|
^{s_{3}} \right)  ^{\frac{s_{2}}{s_{3}}} \right)  ^{\frac{s_{1}}{s_{2}}}
\right)  ^{\frac{1}{s_{1}}} \leq C \prod_{k=1}^{3} \left\|  x^{k}\right\|
_{w,q_{k}},
\end{align*}
once that $v_{1} \in\Pi_{(s_{1};q_{1})}\left(  X_{1};\ell_{(s_{2},s_{3}%
)}(Y)\right)  $. Therefore, $T \in\Pi_{(\mathbf{s};\mathbf{q})}^{3} \left(
X_{1}, X_{2}, X_{3}; Y\right)  $.

Now we shall conclude the proof by an induction argument. Let us suppose the result is true
for $m-1$ and let $T\in\Pi_{(r;\mathbf{p})}^{m}\left(  X_{1},\dots
,X_{m};Y\right)  $, i.e.,
\[
\left(  \sum_{j_{2}=1}^{\infty}\left(  \ldots\left(  \sum_{j_{m}=1}^{\infty
}\left(  \sum_{j_{1}=1}^{\infty}\left\Vert T\left(  x_{\mathbf{j}}\right)
\right\Vert ^{r}\right)  ^{\frac{1}{r}\cdot r}\right)  ^{\frac{1}{r}\cdot
r}\dots\right)  ^{\frac{1}{r}\cdot r}\right)  ^{\frac{1}{r}}=\left(
\sum_{j_{1},\dots,j_{m}=1}^{\infty}\left\Vert T\left(  x_{\mathbf{j}}\right)
\right\Vert ^{r}\right)  ^{\frac{1}{r}}\leq C\prod_{k=1}^{m}\left\Vert
x^{k}\right\Vert _{w,p_{k}},
\]
for all sequences $x^{k}\in\ell_{p_{k}}^{w}(X_{k})$.
For a fixed $x^{1}\in\ell_{p_{1}}^{w}(X_{1})$, $v:X_{2}\times\cdots\times
X_{m}\rightarrow\ell_{r}(Y)$ given by
\[
v(x_{2},\ldots,x_{m}):=\left(  T(x_{j_{1}}^{1},x_{2},\ldots,x_{m})\right)
_{j_{1}\in\mathbb{N}},
\]
belongs to $\Pi_{(r;p_{2},\dots,p_{m})}^{m-1}\left(  X_{2},\dots,X_{m}%
;\ell_{r}(Y)\right)  $. Consequently, by induction hypothesis, norm inclusion
and the Minkowski inequality, {\footnotesize
\begin{align}
\left(  \sum_{j_{1}=1}^{\infty}\left(  \sum_{j_{2}=1}^{\infty}\left(
\ldots\left(  \sum_{j_{m}=1}^{\infty}\left\Vert T\left(  x_{\mathbf{j}%
}\right)  \right\Vert ^{s_{m}}\right)  ^{\frac{s_{m-1}}{s_{m}}}\dots\right)
^{\frac{s_{2}}{s_{3}}}\right)  ^{\frac{s_{2}}{s_{2}}}\right)  ^{\frac{1}%
{s_{2}}}  &  \leq\left(  \sum_{j_{2}=1}^{\infty}\left(  \ldots\left(
\sum_{j_{m}=1}^{\infty}\left(  \sum_{j_{1}=1}^{\infty}\left\Vert T\left(
x_{\mathbf{j}}\right)  \right\Vert ^{r}\right)  ^{\frac{s_{m}}{r}}\right)
^{\frac{s_{m-1}}{s_{m}}}\dots\right)  ^{\frac{s_{2}}{s_{3}}}\right)
^{\frac{1}{s_{2}}}\nonumber\\
&  \leq C\left\Vert x^{1}\right\Vert _{w,p_{1}}\prod_{k=2}^{m}\left\Vert
x^{k}\right\Vert _{w,q_{k}}. \label{Op.sq.summing}%
\end{align}
} with $r\leq s_{m}\leq\dots\leq s_{2}$ and
\[
\frac{1}{s_{2}}=\frac{1}{r}-\sum_{k=2}^{m}\frac{1}{p_{k}}+\sum_{k=2}^{m}%
\frac{1}{q_{k}}.
\]
Fixing $x^{k}\in\ell_{q_{k}}^{w}(X_{k}),\,k=2,\dots,m$ and defining, for all
$x_{1}\in X_{1}$,
\[
u(x_{1})=\left(  T\left(  x_{1},x_{j_{2}}^{2}\dots,x_{j_{m}}^{m}\right)
\right)  _{j_{2},\dots,j_{m}\in\mathbb{N}}%
\]
we have that $u\in\Pi_{(s_{2};p_{1})}\left(  X_{1};\ell_{(s_{2},\dots,s_{m}%
)}(Y)\right)  $. Applying the classical linear inclusion on
(\ref{Op.sq.summing}), with $q_{1}\geq p_{1}$ and $s_{1}\geq s_{2}$ such that
\[
\frac{1}{p_{1}}-\frac{1}{s_{2}}\leq\frac{1}{q_{1}}-\frac{1}{s_{1}},
\]
we gain $u\in\Pi_{(s_{1};q_{1})}\left(  X_{1};\ell_{(s_{2},\dots,s_{m}%
)}(Y)\right)  $. Taking $\frac{1}{s_{1}}=\frac{1}{s_{2}}-\frac{1}{p_{1}}%
+\frac{1}{q_{1}}=\frac{1}{r}-\left\vert \frac{1}{\mathbf{p}}\right\vert
+\left\vert \frac{1}{\mathbf{q}}\right\vert >0$, since $u\in\Pi_{(s_{1}%
;q_{1})}\left(  X_{1};\ell_{(s_{2},\dots,s_{m})}(Y)\right)  $, we have
\[
\left(  \sum_{j_{1}=1}^{\infty}\left(  \ldots\left(  \sum_{j_{m}=1}^{\infty
}\left\Vert T\left(  x_{\mathbf{j}}\right)  \right\Vert ^{s_{m}}\right)
^{\frac{s_{m-1}}{s_{m}}}\dots\right)  ^{\frac{s_{1}}{s_{2}}}\right)
^{\frac{1}{s_{1}}}\leq C\prod_{k=1}^{m}\left\Vert x^{k}\right\Vert _{w,q_{k}%
}.
\]
Therefore, $T\in\Pi_{(\mathbf{s};\mathbf{q})}^{m}\left(  X_{1},\dots
,X_{m};Y\right)  $. Also note that the inclusion operator has norm 1, since
the constant C is preserved. This concludes the proof. \qed

It is important to highlight the difference between Theorems \ref{Pellegrino},
\ref{Bayart} and \ref{Thmain}. Under the hypothesis of Theorem \ref{Thmain},
by using the usual inclusion of $\ell_{p}$ spaces and Theorem \ref{Bayart}
with $\frac{1}{s_{1}}:=\frac{1}{r}-\left\vert \frac{1}{\mathbf{p}}\right\vert
+\left\vert \frac{1}{\mathbf{q}}\right\vert $ one may get that
\[
\Pi_{(\mathbf{r};\mathbf{p})}^{m}\subset\Pi_{(s_{1};\mathbf{q})}^{m}.
\]
But if $r\leq s_{m}\leq\cdots\leq s_{1}$,
\[
\Pi_{(\mathbf{s};\mathbf{q})}^{m}\subset\Pi_{(s_{1};\mathbf{q})}^{m}.
\]
Nevertheless we may not conclude by Theorem \ref{Bayart} that
\[
\Pi_{(r;\mathbf{p})}^{m}\subset\Pi_{(\mathbf{s};\mathbf{q})}^{m}.
\]
However this is provided by Theorem \ref{Thmain}. For instance, let us illustrate with a numerical example: let $r:=3,\,\mathbf{s}:=(5,3),\,\mathbf{p}:=(3,2)$ and
$\mathbf{q}:=(5,2)$. From Theorem \ref{Bayart} we have
\[
\Pi_{(3;\mathbf{p})}^{2}\subset\Pi_{(5;\mathbf{q})}^{2},
\]
while Theorem \ref{Thmain} provides
\[
\Pi_{(3;\mathbf{p})}^{2}\subset\Pi_{(5,3;\mathbf{q})}^{2}.
\]
The same can also be done when $m=3$: let $r:=2,\,\mathbf{s}%
:=(6,3,2),\,\mathbf{p}:=(2,2,1)$ and $\mathbf{q}:=(3,3,1)$. Then
\[
\Pi_{(2;\mathbf{p})}^{3}\subset\Pi_{(6,3,2;\mathbf{q})}^{3}\subset
\Pi_{(6;\mathbf{q})}^{3},
\]
where the first inclusion is assured by Theorem \ref{Thmain}.

\section{A new Regularity Principle for sequence spaces}

The investigation of regularity-type results in this setting was initiated in
\cite{pellv} and  expanded in \cite{PSST}. In this short section we present
a stronger version of these results. 

Let $m\geq2$ and $Z_{1}$, $V$ and $w_{1},\dots,W_{m}$ be arbitrary non-empty
sets and $Z_{2},\dots,Z_{m}$ be vector spaces. Let also
\[
R_{k}:Z_{k}\times W_{k}\rightarrow\lbrack0,\infty)\quad\text{and}\quad
S:Z_{1}\times\cdots\times Z_{m} \times V\rightarrow\lbrack0,\infty)
\]
be arbitrary maps, with $k=1,\dots,m,$ satisfying
\[
R_{k}(\lambda z,w)=\lambda R_{k}(z,w)\quad and\quad S(z_{1},\dots
,z_{j-1},\lambda z_{j},z_{j+1},\dots,z_{m},\nu)=\lambda S(z_{1},\dots
,z_{j-1},z_{j},z_{j+1},\dots,z_{m},\nu)
\]
for all scalars $\lambda\geq0$ and $j,k \in \{2,\dots,m\}$. We shall work with each $p_{k} \geq 1$ and also assuming that
\[
\sup_{w\in W_{k}} \left(\sum_{j=1}^{n_{k}}R_{k} \left(  z_{j}^{k}, w\right)^{p_{k}}\right)^{\frac{1}{p_{k}}} < \infty,
\quad k=1,\dots,m.
\]
Despite the abstract context, the proof is similar to the the proof of Theorem \ref{Thmain}, and we omit the details.

\begin{theorem}
[Anisotropic Regularity Principle]Let $m$ be a positive integer,
$r\geq1,\,\mathbf{s},\mathbf{p},\mathbf{q}\in\lbrack1,+\infty)^{m}$ be such
that $q_{k}\geq p_{k}$, for $k=1,\dots,m$ and
\[
\frac{1}{r}-\left\vert \frac{1}{\mathbf{p}}\right\vert +\left\vert \frac
{1}{\mathbf{q}}\right\vert >0.
\]
Assume that there exists a constant $C > 0$ such that
\[
\sup_{\nu\in V}\left(  \sum\limits_{j_{1}=1}^{n_{1}}\cdots\sum\limits_{j_{m}%
=1}^{n_{m}}S\left(  z_{j_{1}}^{1},\ldots,z_{j_{m}}^{m},\nu\right)
^{r}\right)  ^{\frac{1}{r}}\leq C\cdot\prod_{k=1}^{m}\sup_{w\in W_{k}} \left(  \sum_{j=1}^{n_{k}}R_{k} \left(  z_{j}^{k}, w\right)^{p_{k}}\right)^{\frac{1}{p_{k}}},
\]
for all $z_{j}^{(k)}\in Z_{k}$ and $n_{k}\in\mathbb{N}$ with $k=1,\dots,m$.
Then
\[
\sup_{\nu\in V}\left(  \sum\limits_{j_{1}=1}^{n_{1}}\left(  \cdots\left(
\sum\limits_{j_{m}=1}^{n_{m}}S\left(  z_{j_{1}}^{1},\ldots,z_{j_{m}}^{m}%
,\nu\right)  ^{s_{m}}\right)  ^{\frac{s_{m-1}}{s_{m}}}\cdots\right)
^{\frac{s_{1}}{s_{2}}}\right)  ^{\frac{1}{s_{1}}}\leq C\cdot\prod_{k=1}%
^{m}\sup_{w \in W_{k}}\left(  \sum_{j=1}^{n_{k}}R_{k}\left(  z_{j}^{k}, w\right)^{q_{k}} \right)^{\frac{1}{q_{k}}},
\]
for all $z_{j}^{(k)}\in Z_{k}$ and $n_{k}\in\mathbb{N},\,k=1,\dots,m$, with
\[
\frac{1}{s_{k}} - \left|  \frac{1}{\mathbf{q}}\right|  _{\geq k} = \frac{1}{r}
- \left|  \frac{1}{\mathbf{p}}\right|_{\geq k},
\quad k \in\{1, \dots, m\}.
\]
\end{theorem}

\section{Applications: Hardy--Littlewood's inequalities}

The Hardy--Littlewood inequalities have been investigated in depth in the
recent years (see, for instance,
\cite{abpsjfa,abpshl,araujo.jfa,APNSHL,cav.pos,cav.quaest,dimant,maia,nunes,PSST,
PT}). Here $\mathcal{X}_{p}$ stands for $\ell_{p}$ if $p<\infty$ and
$\mathcal{X}_{\infty}:=c_{0}$.

\begin{theorem}[Albuquerque, Bayart, Pellegrino, Seoane \cite{abpshl}, 2014]\label{BH_HL}
Let $\mathbf{p},\mathbf{s}\in\lbrack1,+\infty]^{m}$ such that $\left\vert
1/\mathbf{p}\right\vert \leq\frac{1}{2}$ and $\mathbf{s}\in\left[  \left(
1-\left\vert 1/\mathbf{p}\right\vert \right)  ^{-1},2\right]  ^{m}$. There is
a constant $C_{m,\mathbf{p},\mathbf{s}}^{\mathbb{K}}\geq1$ such that
\[
\left(  \sum\limits_{j_{1}=1}^{\infty}\left(  \cdots\left(  \sum
\limits_{j_{m}=1}^{\infty}\left\vert A\left(  e_{j_{1}},\ldots,e_{j_{m}%
}\right)  \right\vert ^{s_{m}}\right)  ^{\frac{s_{m-1}}{s_{m}}}\cdots\right)
^{\frac{s_{1}}{s_{2}}}\right)  ^{\frac{1}{s_{1}}}\leq C_{m,\mathbf{p}%
,\mathbf{s}}^{\mathbb{K}}\left\Vert A\right\Vert \label{ttr}%
\]
for every continuous $m$-linear form $A:\mathcal{X}_{p_{1}}\times\cdots
\times\mathcal{X}_{p_{m}}\rightarrow\mathbb{K}$ if, and only if,
\[
\left\vert \frac{1}{\mathbf{s}}\right\vert \leq\frac{m+1}{2}-\left\vert
{\frac{1}{\mathbf{p}}}\right\vert .
\]

\end{theorem}

\begin{theorem}
[Dimant, Sevilla-Peris \cite{dimant}, 2016]\label{78} Let $\mathbf{p}%
\in\lbrack1,+\infty]^{m}$ such that $1/2\leq\left\vert 1/\mathbf{p}\right\vert
<1$. There exists a (optimal) constant $D_{m,\mathbf{p}}^{\mathbb{K}}\geq1$
such that
\begin{equation}
\left(  \sum_{j_{1},\dots,j_{m}=1}^{\infty}\left\vert A(e_{j_{1}}%
,\dots,e_{j_{m}})\right\vert ^{\frac{1}{1-\left\vert 1/\mathbf{p}\right\vert
}}\right)  ^{1-\left\vert 1/\mathbf{p}\right\vert }\leq D_{m,\mathbf{p}%
}^{\mathbb{K}}\Vert A\Vert\label{qz}%
\end{equation}
For every continuous $m$-linear form $A:\mathcal{X}_{p_{1}}\times\cdots
\times\mathcal{X}_{p_{m}}\rightarrow\mathbb{K}$. Moreover, the exponent is
optimal. \label{DSP}
\end{theorem}

The above exponent $\frac{1}{1-\left\vert 1/\mathbf{p}\right\vert }$ is
optimal, but not in the anisotropic sense. In \cite[Theorem 3.2]{APNSHL} the
authors improved Theorem \ref{78}, under some restriction over $\mathbf{p}$.
The following recent result is a kind of anisotropic extension of it:

\begin{theorem}
[Aron, N\'{u}\~{n}ez, Pellegrino, Serrano \cite{APNSHL}, 2017]\label{ThANPS}
Let $m\geq2,q_{1},\dots,q_{m}>0$, and $1<p_{m}\leq2<p_{1},\dots,p_{m-1}$, with
$\left\vert 1/\mathbf{p}\right\vert <1$. There is a (optimal) constant
$C_{\mathbf{p}}\geq1$ such that
\[
\left(  \sum\limits_{j_{1}=1}^{\infty}\left(  \cdots\left(  \sum
\limits_{j_{m}=1}^{\infty}\left\vert A\left(  e_{j_{1}},\ldots,e_{j_{m}%
}\right)  \right\vert ^{s_{m}}\right)  ^{\frac{s_{m-1}}{s_{m}}}\cdots\right)
^{\frac{s_{1}}{s_{2}}}\right)  ^{\frac{1}{s_{1}}}\leq C_{\mathbf{p}}\left\Vert
A\right\Vert
\]
for every continuous $m$-linear form $A:\mathcal{X}_{p_{1}}\times\cdots
\times\mathcal{X}_{p_{m}}\rightarrow\mathbb{K}$ if, and only if,
\[
s_{k}\geq\left[  1-\left\vert \frac{1}{\mathbf{p}}\right\vert _{\geq
k}\right]  ^{-1},\quad\text{ for }k=1,\dots,m.
\]
\end{theorem}

Recently, W.V. Cavalcante has shown that (\ref{qz}) is a consequence of the inclusion result for multiple summing operators due to Pellegrino \textit{et al.} combined with Theorem \ref{BH_HL} (see \cite{cav.pos}). The standard isometries between $\mathcal{L}(\mathcal{X}_{p},X)$ and $\ell_{p^{\ast}}^{w}(X)$, for $1<p\leq\infty$, allow us to read the previous Theorems \ref{BH_HL}, \ref{78}, \ref{ThANPS} as coincidence results (see \cite{DJT}). The key point is to begin with the coincidence below, obtained by revisiting Theorem \ref{BH_HL} as a coincidence result with $s_1=\cdots=s_m=2$ and $p_1=\cdots = p_m = 2m$,
\[
\Pi_{\left(2; (2m)^{\ast}\right)}^{m} \left(  X_{1},\dots,X_{m}; \mathbb{K}\right)
=\mathcal{L}\left(  X_{1},\dots,X_{m};\mathbb{K}\right),
\]
for all Banach spaces $X_1,\dots,X_m$, and use an inclusion-type result. We shall combine these isometries with the inclusion result Theorem \ref{Thmain} to gain refined inclusions and coincidences.

\begin{theorem}\label{new_HL_coin}
Let $m$ be a positive integer and $\mathbf{s}%
,\,\mathbf{p}\in\lbrack1,+\infty]^{m}$. If $\left\vert 1/\mathbf{p}\right\vert
<1$ and $p_{1},\dots,p_{m}\leq2m$, then
\[
\Pi_{\left(  2;(2m)^{\ast}\right)  }^{m}\left(  X_{1},\dots,X_{m} ;
\mathbb{K}\right)  \subset\Pi_{\left(  \mathbf{s};p_{1}^{\ast},\ldots
,p_{m}^{\ast}\right)  }^{m}\left(  X_{1},\dots,X_{m};\mathbb{K}\right)  ,
\]
for any Banach spaces $X_{1},\dots,X_{m}$, with
\[
s_{k}=\left[  \frac{1}{2}+\frac{m-k+1}{2m} - \left|  \frac{1}{\mathbf{p}}
\right|  _{\geq k}\right]  ^{-1}, \quad\text{ for } k=1,\dots,m.
\]

\end{theorem}

\begin{proof}
Since each $p_{k}\leq2m$ and
\[
\frac{1}{2}-m\cdot\left(  1-\frac{1}{2m}\right)  +m-\left\vert \frac
{1}{\mathbf{p}}\right\vert =1-\left\vert \frac{1}{\mathbf{p}}\right\vert >0,
\]
applying Inclusion Theorem \ref{Thmain}, it is obtained the stated inclusion
with
\[
\frac{1}{s_{k}}=\frac{1}{2}-(m-k+1)\cdot\left(  1-\frac{1}{2m}\right)
+(m-k+1)-\left\vert \frac{1}{\mathbf{p}}\right\vert _{\geq k},
\]
for each $k=1,\dots,m$.
\end{proof}

\begin{corollary}
If $\left\vert 1/\mathbf{p}\right\vert <1$ and $p_{1},\dots,p_{m}\leq2m$,
then
\[
\Pi_{\left(  \mathbf{s};p_{1}^{\ast},\ldots,p_{m}^{\ast}\right)  }^{m}\left(
X_{1},\dots,X_{m};\mathbb{K}\right)  =\mathcal{L}\left(  X_{1},\dots
,X_{m};\mathbb{K}\right)  ,
\]
for any Banach spaces $X_{1},\dots,X_{m}$ and
\[
\frac{1}{s_{k}} - \left|  \frac{1}{\mathbf{q}}\right|  _{\geq k} = \frac{1}{r}
- \left|  \frac{1}{\mathbf{p}}\right|_{\geq k},
\quad k \in\{1, \dots, m\}.
\]
\end{corollary}


Bringing Theorem \ref{new_HL_coin} to the context of sequence spaces, the
announced anisotropic result will be achieved. The main result of this section
reads as follows.

\begin{corollary}
\label{new_HL} Let $m$ be a positive integer and $\mathbf{p}\in\lbrack
1,+\infty)^{m}$ such that $\left\vert 1/\mathbf{p}\right\vert <1$ and
$p_{1},\dots,p_{m}\leq2m$. Then, for all continuous $m$-linear forms
$A:\ell_{p_{1}}\times\cdots\times\ell_{p_{m}}\rightarrow\mathbb{K}$
\begin{equation}
\left(  \sum_{j_{1}=1}^{\infty}\left(  \ldots\left(  \sum_{j_{m}=1}^{\infty
}\left\Vert A\left(  e_{j_{1}},\dots,e_{j_{m}}\right)  \right\Vert ^{s_{m}%
}\right)  ^{\frac{s_{m-1}}{s_{m}}}\dots\right)  ^{\frac{s_{1}}{s_{2}}}\right)
^{\frac{1}{s_{1}}}\leq D_{m,p,\mathbf{s}}^{\mathbb{K}}\Vert A\Vert
\label{new_exp}%
\end{equation}
with
\[
s_{k}=\left[  \frac{1}{2}+\frac{m-k+1}{2m}-\left\vert \frac{1}{\mathbf{p}%
}\right\vert _{\geq k}\right]  ^{-1},\quad\text{ for }k=1,\dots,m.
\]

\end{corollary}

In order to clarify the new result, we illustrate the simpler case: when
dealing with $m<p_{1}=\cdots=p_{m}=p\leq2m$, we get \eqref{new_exp} with
exponents
\[
s_{k}=\left[  \frac{1}{2}-(m-k+1)\cdot\left(  \frac{1}{p}-\frac{1}{2m}\right)
\right]  ^{-1},\quad\text{ for }k=1,\dots,m,
\]
that is,
\[
s_{1}=\frac{p}{p-m},
\dots,s_{m}=\frac{2mp}{mp+p-2m}.
\]
It is obvious that the above exponents are better than the estimates of
Theorem \ref{78} that provides
\[
s_{1}=\cdots=s_{m}=\frac{p}{p-m}.
\]

The following example is illustrative:

\begin{example}
Suppose $m=3$ and $p=4$. By Theorem \ref{78} we know that \eqref{new_exp}
holds with
\[
s_{k} \geq4 \text{ for } k=1,2,3,
\]
whereas by Corollary \ref{new_HL} we have
\[
s_{1} \geq4, \quad s_{2} \geq3 \quad\text{ and } \quad s_{3} \geq 12/5.
\]

\end{example}


\begin{thebibliography}{99}                                                                                               %
	\bibitem {abpsjfa}N. Albuquerque, F. Bayart, D. Pellegrino and J.B.
	Seoane-Sep\'{u}lveda, Sharp generalizations of the multilinear
	Bohnenblust--Hille inequality, J. Funct. Anal. \textbf{266} (2014), 3726--3740.
	
	\bibitem {abpshl}N. Albuquerque, F. Bayart, D. Pellegrino and J.B.
	Seoane-Sep\'{u}lveda, Optimal Hardy--Littlewood type inequalities for
	polynomials and multilinear operators, Israel J. Math. \textbf{211} (2016),
	no. 1, 197--220.
	
	\bibitem {araujo.jfa}G. Ara\'{u}jo, D. Pellegrino and D. D. P. da S. e Silva,
	On the upper bounds for the constants of the Hardy--Littlewood inequality, J.
	Funct. Anal. \textbf{267} (2014), no. 6, 1878--1888.
	
	\bibitem {APNSHL}R. Aron, D. Pellegrino, D. N\'{u}\~{n}ez-Alarc\'{o}n and D.
	Serrano-Rodr\'{\i}guez, Optimal exponents for Hardy--Littlewood inequalities
	for $m$-linear operators, Linear Algebra and its App., in press.
	
	\bibitem {bayart}F. Bayart, Multiple summing maps; coordinatewise summability,
	inclusion theorems and $p$-Sidon sets, arXiv:1704.04437v1.
	
	\bibitem {benedek}A. Benedek, R. Panzone,\emph{T}he space $L_{\mathbf{p}}$, with mixed norm, Duke Math. J. \textbf{28} (1961) 301--324.
	
	\bibitem {blasco}O. Blasco, G. Botelho, D.\ Pellegrino and P. Rueda, Coincidence results for summing multilinear mappings. Proc. Edinb. Math. Soc. (2) 
	
	\bibitem {pams}G. Botelho, H.-A. Braunss, H. Junek and D. Pellegrino, \emph{Inclusions and coincidences for multiple summing multilinear mappings}, Proc. Amer. Math. Soc. \textbf{137} (2009), no. 3, 991--1000.
	
	\bibitem {complu}G. Botelho, C. Michels, H. Junek and D. Pellegrino,
	\emph{Complex interpolation and summability properties of multilinear
		operators}. Rev. Mat. Complut. \textbf{23} (2010), no. 1, 139--161.
	
\bibitem{cav.pos} W.V. Cavalcante, Some applications of the Regularity Principle in sequence spaces, Positivity, in press.
	
	\bibitem {cav.quaest}W.V. Cavalcante, D.N\'{u}\~{n}ez-Alarc\'{o}n, Remarks on
	an inequality of Hardy and Littlewood. Quaest. Math. \textbf{39} (2016), no.
	8, 1101--1113.
	
	\bibitem {DJT}J. Diestel, H. Jarchow, and A. Tonge, Absolutely summing
	operators, Cambridge Studies in Advanced Mathematics 43, Cambridge University
	Press, Cambridge, 1995.
	
	\bibitem {dimant}V. Dimant and P. Sevilla-Peris, Summation of coefficients of
	polynomials on $\ell_{p}$ spaces, Publ. Mat. \textbf{60} (2016), no. 2, 289--310.
	
	\bibitem {G}D. J.H. Garling, Inequalities: a journey into linear analysis,
	Cambridge University Press, Cambridge, 2007.
	
	\bibitem {hardy}G. Hardy and J.E. Littlewood, Bilinear forms bounded in space
	$[p,q]$, Quart. J. Math. \textbf{5} (1934), 241--254.
	
	\bibitem {l}J.E. Littlewood, On bounded bilinear forms in an infinite number
	of variables, Quart. J. Math. Oxford, \textbf{1} (1930), 164--174.
	
	\bibitem {maia}M. Maia, T. Nogueira, D. Pellegrino, The Bohnenblust--Hille
	inequality for polynomials whose monomials have a uniformly bounded number of
	variables, Integr. Equ. Oper. Theory 88 (2017), 143--149.
	
	\bibitem {matos03}M.C. Matos, \emph{Fully absolutely summing and Hilbert-Schmidt multilinear mappings}, Collectanea Math. \textbf{54} (2003), 111--136.
	
	\bibitem {nunes}A. Nunes, A new estimate for the constants of an inequality
	due to Hardy and Littlewood. Linear Algebra Appl. 526 (2017), 27--34.
	
	\bibitem {pellv} D. Pellegrino, J. Santos. J.B. Seoane-Sepulveda, Some techniques on nonlinear analysis and applications. Adv. Math. 229 (2012), no. 2, 1235--1265.
	
	\bibitem {PSST}D. Pellegrino, J. Santos, D. Serrano-Rodr\'{\i}guez, E.
	Teixeira, Regularity principle in sequence spaces and applications, arXiv:1608.03423v2
	
	\bibitem {PT}D. Pellegrino, E. Teixeira, Towards sharp Bohnenblust–Hille constants, Comm. Contemp. Math, in press.
	
	\bibitem {pg}D. P\'{e}rez-Garc\'{\i}a, Operadores multilineales absolutamente
	sumantes, Ph.D. Thesis, Universidad Complutense de Madrid, 2004.
	
	\bibitem{pra} T. Praciano--Pereira, On bounded multilinear forms on a class
	of $\ell _{p}$ spaces, J. Math. Anal. Appl. \textbf{81} (1981), 561--568.
	
\end{thebibliography}
\end{document}